\documentclass[12pt,twoside]{article}

\usepackage[ansinew]{inputenc}

\usepackage{makeidx}
\usepackage{textcomp}

\usepackage{amsmath,amsfonts,amssymb,amsthm,amscd,mathrsfs}
\usepackage[colorlinks=true,linkcolor=blue,citecolor=blue,urlcolor=blue]{hyperref}

\usepackage{tikz}
\tikzstyle{vertex}=[circle, draw, inner sep=0pt, minimum size=6pt]

\usepackage{pdfpages}
\usepackage{relsize}
\usepackage{hyperref}
\usepackage{graphicx}

\usepackage{bbm}
\usepackage{color}

\usepackage{url}
\usepackage{epsfig}
\usepackage{longtable}
\usepackage{subfigure}
\usepackage{multirow}

\usepackage{bbm}

\usepackage{setspace}

\newtheorem{thm}{Theorem}[section]
\newtheorem{cor}[thm]{Corollary}

\newtheorem{prop}[thm]{Proposition}

\numberwithin{equation}{section}

\theoremstyle{definition}

\theoremstyle{definition}
\newtheorem{example}[thm]{Example}

\theoremstyle{remark}
\newtheorem{remark}[thm]{\bf Remark}

\newcommand{\s}{\mathbb{S}}

\newcommand{\Cov}{\textrm{Cov}}

\def\E{\mathbb{E}}

\begin{document}

\title{\Large\textbf{Generalized Fr\'echet Bounds for Cell Entries in Multidimensional Contingency Tables}}

\author{{Caroline Uhler}\,\thanks{
Laboratory for Information and Decision Systems, Department of Electrical Engineering and Computer Science, Institute for Data, Systems and Society, Massachusetts Institute of Technology, Cambridge, MA 02139, U.S.A. E-mail address: \href{mailto:cuhler@mit.edu}{cuhler@mit.edu}.
\endgraf
\ $^\dag$Department of Statistics, Pennsylvania State University, University Park, PA 16802, U.S.A. E-mail address: \href{mailto:richards@stat.psu.edu}{richards@stat.psu.edu}.
\endgraf
\endgraf
\endgraf
}\ \ and {Donald Richards}\,$^\dag$
 \endgraf
}

\date{\today}

\maketitle

\begin{abstract}
We consider the lattice, $\mathcal{L}$, of all subsets of a multidimensional contingency table and establish the properties of monotonicity and supermodularity for the marginalization function, $n(\cdot)$, on $\mathcal{L}$.  We derive from the supermodularity of $n(\cdot)$ some generalized Fr\'echet inequalities complementing and extending inequalities of Dobra and Fienberg.  Further, we construct new monotonic and supermodular functions from $n(\cdot)$, and we remark on the connection between supermodularity and some correlation inequalities for probability distributions on lattices.  We also apply an inequality of Ky Fan to derive a new approach to Fr\'echet inequalities for multidimensional contingency tables.
\end{abstract}

\section{Introduction}\label{sec:introduction}

In the statistical analysis of contingency tables, the derivation of upper and lower bounds for cell entries has been accorded considerable attention.  The motivation for this problem stems from a broad range of areas, including statistical inference in the social and biomedical sciences, ecology, computer-aided tomography, causal analysis, graphical models, survey sampling, privacy and disclosure limitation, observational studies, and other fields.  We refer to Dobra \cite{dobrathesis}, Dobra and Fienberg \cite{dobra2000bounds}, and Fienberg \cite{fienberg1999frechet} for detailed accounts of results in this area and numerous references to the literature.  

The motivation for this paper stems from the work of Dobra and Fienberg \cite{dobra2000bounds}, who derived for the cell entries of multidimensional contingency tables a class of generalizations of the classical inequalities of Fr\'echet \cite{frechet1940probabilites}.  We were especially intrigued by the possibility of developing an approach to these inequalities complementing the graph-theoretic treatment given in \cite{dobra2000bounds}.  

In deriving our results, we also apply an inequality of Ky Fan \cite{fan1968inequality} which seems to have been overlooked hitherto within the literature on supermodularity.  We deduce from Fan's inequality some known Fr\'echet and Boole inequalities, derived by Fienberg~\cite{fienberg1999frechet} for multiway contingency tables, and we apply Fan's inequalities to derive Fr\'echet inequalities for general multiway contingency tables free of graphical restrictions arising from loglinear models.  We also investigate the limitations of Fan's inequality by showing that, in at least one instance, the inequality provides a bound which is weaker than a coresponding bound which we obtain from the results of Dobra and Fienberg \cite{dobra2000bounds}.  

Our results are as follows.  We establish in Section \ref{sec:generalizedfrechet} the monotonicity and supermodularity of the marginalization function, $n(\cdot)$, of a multidimensional contingency table.  We deduce Fr\'echet inequalities from the supermodularity of $n(\cdot)$, develop new monotonic and supermodular functions from $n(\cdot)$, and remark on the connection between supermodularity and correlation inequalities for certain probability distributions on contingency tables.  In Section \ref{sec:fan}, we apply Fan's inequality, thereby obtaining a new approach to deriving Fr\'echet inequalities for multidimensional contingency tables.  In Section \ref{sec:discussion}, we remark on a general procedure for interpreting classes of correlation inequalities for log-supermodular probability density functions as Fr\'echet inequalities.

\section{Generalized Fr\'echet bounds}\label{sec:generalizedfrechet}

Let $L=\{1,\dots, \ell\}$ be an index set, and denote by $\mathcal{L}$ the set of all subsets of $L$.  Then $\mathcal{L}$ is partially ordered by set-theoretic inclusion $\subset$ and forms a complete finite distributive lattice, where the meet $\wedge$ and join $\vee$ operations coincide with the set-theoretic operations of intersection $\cap$ and union $\cup$, respectively. 

Let $X_1,\ldots,X_\ell$ be discrete random variables.  We suppose that each $X_j$ takes values $x_j \in I_j$, a discrete set of labels, $j=1,\ldots,\ell$.  We define the discrete random vector $X = (X_1,\dots X_\ell)$, whose values are $x = (x_1,\ldots,x_\ell) \in J_L = I_1\times \cdots \times I_\ell$.  

Consider an {\it $\ell$-way contingency table} $n:=(n_x: x\in J_L)$. For each collection of labels $a=\{i_1,\ldots,i_p\} \subset L$, let $J_a:=I_{i_1}\times \cdots \times I_{i_p}$; then each $x \in J_L$ can be written in the form $x = (x_a,x_{L\setminus a})$.  Define the \emph{marginalization function}, 
\begin{equation}
\label{eq_marginalization_function}
n_{x(a),+} :=\sum_{x_{L\setminus a}\in J_{L\setminus a}} n_{x_a,\, x_{L\setminus a}},
\end{equation}
$a \subset L$.  Consequently, it can be seen that each contingency table $n$ defines a marginalization function that takes input $a \subset L$ and outputs a \emph{marginal table} $n(a) := (n_{x(a),+} : x(a)\in J_a)$. In this way, we can identify the marginalization function with the contingency table.  

As an example, consider the 2-way contingency table in Table \ref{leadstudy}, arising in a well-known study~\cite{morton1982lead}, \cite[p.~81~ff.]{rosenbaum2002observational} of the presence of lead in the blood of children of employees in an industrial factory in Oklahoma which used lead in the manufacture of batteries.
%
\begin{table}[!t]
\caption{Number of children classified by father's hygiene and by father's exposure to lead.}
\medskip
\label{leadstudy}
\centering
\qquad\quad \begin{tabular}{|ll|ccc|}
\hline
&& \multicolumn{3}{|c|}{{Father's exposure}} \\
&& Low & Medium & High \\
\hline
\multirow{3}{*}{{Father's hygiene}} & Poor & 7 & 5 & 13 \\
\!\!\!\!\!\!\!\!\!\!\!\!\!\!\!\!\!\!\!\! $n = $ &Medium & 1 & 1 & \phantom{1}3 \\
&Good   & 0 & 1 & \phantom{1}3 \\
\hline
\end{tabular}
\end{table}
In this example, $\ell=2$, $I_1 = \{\textrm{Poor}, \textrm{Medium}, \textrm{Good}\}$, the levels of father's hygiene, and $I_2 = \{\textrm{Low}, \textrm{Medium}, \textrm{High}\}$, the levels of father's exposure. Then $n(\{1,2\})=n$ denotes the contingency table itself, $n_{i,j}$ denotes the $(i,j)$-th entry of the contingency table corresponding to $x_1=i$ and $x_2=j$, and $n(\emptyset) = 34$ is the total number of individuals in the table. Also, the marginal tables are the row sums, $n(\{1\}) = (25,5,4)$, and the column sums, $n(\{2\}) = (8,7,19)$. 

In order to simplify the marginalization notation, we may write, for example, $n_{+,j}$ instead of $n(\{2\})_j$.  Then, 
according to the \emph{simple Fr\'echet bound} \cite{frechet1940probabilites}, each cell entry in a 2-way table is bounded by the 1-way marginals in the following way:
$$
\min\big(n(\{1\})_{i}, n(\{2\})_{j}\big) \,\geq\, n(\{1,2\})_{i,j} \,\geq\, \max\big(n(\{1\})_{i}+ n(\{2\})_{j} - n({\emptyset}), 0\big),
$$
for all $(i,j) \in I_1 \times I_2$.  In the simpler notation, this statement is equivalent to 
$$
\min(n_{i,+}, n_{+,j}) \,\geq\, n_{i,j} \,\geq\, \max(n_{i,+}+ n_{+,j} - n_{++}, 0),
$$
for all $(i,j) \in I_1 \times I_2$. Whenever a statement holds for all choices of indices, we will omit the indices; then the simple Fr\'echet bounds are given by
\begin{equation}
\label{eq_simple_Frechet}
\min\big(n(\{1\}), n(\{2\})\big) \,\geq\, n(\{1,2\}) \,\geq\, \max\big(n(\{1\})+ n(\{2\}) - n({\emptyset}), 0\big).
\end{equation}

In the following result, we generalize these bounds to multiway conditional tables. Namely, we prove that the marginalization function $n$ is decreasing and supermodular.

\begin{thm}\label{thm_supermod}
The marginalization function $n$ has the following properties:

\noindent
(a) $n$ is decreasing on $\mathcal{L}$, i.e., $n(a)\geq n(b)$ for all $a\subset b\in \mathcal{L}$, and 

\noindent
(b) $n$ is supermodular on $\mathcal{L}$, i.e., 
\begin{equation}
\label{eq_supermodular}
n(a\cup b) + n(a\cap b) \geq n(a) + n(b) 
\end{equation}
for all $a,b\in \mathcal{L}$.
\end{thm}

\begin{proof}
Property (a) follows directly from (\ref{eq_marginalization_function}) and the non-negativity of all cell entries in a contingency table. 

To establish Property (b), the supermodularity of $n$, we consider three cases:

\noindent
(1) $a\subset b$ or $b\subset a$: Then $n(a\cup b) + n(a\cap b) = n(a) + n(b)$, so the inequality (\ref{eq_supermodular}) is valid, trivially.

\noindent
(2) $a\cap b = \emptyset$: Without loss of generality, let $a=\{i_1,\dots,i_q\}$ and $b=\{i_{q+1},\dots,i_m\}$ and let $n_{i_1^0,\dots,i_m^0,+}$ denote an arbitrary cell in the marginal contingency table $n(a\cup b)$ corresponding to $n$, where $(i_1^0, \dots , i_q^0)\in J_a$ and $(i_{q+1}^0, \dots , i_m^0)\in J_b$. Then
\begin{align*}
n(\emptyset) &\equiv \sum_{i_1,\dots, i_m} n_{i_1,\dots,i_m,+} \\
&= \sum_{i_{q+1},\dots,i_m} n_{i_1^0,\dots,i_q^0,i_{q+1},\dots,i_m,+} + \sum_{i_{1},\dots,i_q} n_{i_1,\dots,i_q,i_{q+1}^0,\dots,i_m^0,+} -  n_{i_1^0,\dots,i_m^0,+} \\
& \qquad\qquad + \sum_{(i_1,\dots,i_m)\neq (i_1^0,\dots,i_m^0)} n_{i_1,\dots,i_m,+}.
\end{align*}
Discarding the last term, we obtain 
\begin{eqnarray*}
n(\emptyset)&\geq& \sum_{i_{q+1},\dots,i_m} n_{i_1^0,\dots,i_q^0,i_{q+1},\dots,i_m,+} + \sum_{i_{1},\dots,i_q} n_{i_1,\dots,i_q,i_{q+1}^0,\dots,i_m^0,+} -  n_{i_1^0,\dots,i_m^0,+}\\
&=& n(a) + n(b) - n(a\cup b).
\end{eqnarray*}

\noindent(3) For the last case we assume without loss of generality that $a=\{i_1,\dots,i_q\}$ and $b=\{i_{p},\dots,i_m\}$ with $p\leq q$. Similar to the previous case, let $n_{i_1^0,\dots,i_m^0,+}$ denote an arbitrary cell in the marginal contingency table $n(a\cup b)$ corresponding to $n$, where $(i_1^0, \dots , i_q^0)\in J_a$ and $(i_{p}^0, \dots , i_m^0)\in J_b$. Then
\begin{eqnarray*}
n(a\cap b) &=& \sum_{i_1,\dots, i_{p-1},i_{q+1},\dots,i_m} n_{i_1,\dots,i_{p-1},i_p^0,\dots,i_q^0,i_{q+1},\dots,i_m,+}\\
&=& \sum_{i_{q+1},\dots,i_m} n_{i_1^0,\dots,i_q^0,i_{q+1},\dots,i_m,+} + \sum_{i_{1},\dots, i_{p-1}} n_{i_1,\dots,i_{p-1},i_{p}^0,\dots,i_m^0,+} -  n_{i_1^0,\dots,i_m^0,+} \\
& &\quad + \sum_{\substack{(i_1,\dots,i_{p-1})\neq(i_1^0,\dots,i_{p-1}^0), \\ (i_{q+1},\dots,i_{m})\neq(i_{q+1}^0,\dots,i_m^0)}} n_{i_1,\dots,i_{p-1},i_p^0,\dots,i_q^0,i_{q+1},\dots,i_m,+} 
\\
\phantom{n(a\cap b)} &\geq& \sum_{i_{q+1},\dots,i_m} n_{i_1^0,\dots,i_q^0,i_{q+1},\dots,i_m,+} + \sum_{i_{1},\dots,i_{p-1}} n_{i_1,\dots,i_{p-1},i_{p}^0,\dots,i_m^0,+} -  n_{i_1^0,\dots,i_m^0,+}
\end{eqnarray*}
\begin{eqnarray*}
&=& n(a) + n(b) - n(a\cup b).\phantom{\sum_{i_{1},\dots,i_{p-1}} n_{i_1,\dots,i_{p-1},i_{p}^0,\dots,i_m^0,+} -  n_{i_1^0}}
\end{eqnarray*}
This completes the proof.
\end{proof}

Note that the simple Fr\'echet inequalities (\ref{eq_simple_Frechet}) are a corollary of Theorem~\ref{thm_supermod}; namely, the first inequality is a consequence of the property that $n(\cdot)$ is decreasing and the second inequality is a consequence of supermodularity and the non-negativity of the cell entries. We now construct new supermodular functions from the marginalization function $n(\cdot)$. 

\medskip

\begin{prop}\label{prop:incr_supermodular}
The following functions are increasing and supermodular:

\noindent
(a) For $s \in\mathcal{L}$, the indicator function is defined as
$$
\mathbf{1}_s(a) := \mathbf{1}_{\{s\subset a\}} = 
\begin{cases}
1, & \hbox{ if } s\subset a, \\
0, & \hbox{ otherwise,}
\end{cases}
$$
for all $a\in\mathcal{L}$. 

\noindent
(b) The cumulative function is defined as
$$
f(a) := \sum_{s\in\mathcal{L}} \mathbf{1}_s(a) \,n(s) = \sum_{s: s\subset a} n(s), 
\qquad a\in\mathcal{L}.
$$
\end{prop}

\begin{proof} 
(a) It is clear that $\mathbf{1}_s(a)$ is increasing. So we need to prove that 
$$
\mathbf{1}_{\{s\subset a\}} + \mathbf{1}_{\{s\subset b\}}\leq \mathbf{1}_{\{s\subset a\cup b\}} + \mathbf{1}_{\{s\subset a\cap b\}}.
$$
We analyze the inequality in three cases: First, if $s\not\subset a$ and $s\not\subset b$, then  
$$
\mathbf{1}_{\{s\subset a\}} + \mathbf{1}_{\{s\subset b\}} = 0\leq \mathbf{1}_{\{s\subset a\cup b\}} + \mathbf{1}_{\{s\subset a\cap b\}}.
$$
Second, if $s\subset a$ but $s\not\subset b$ or if  $s\subset b$ but $s\not\subset a$, then 
$$
\mathbf{1}_{\{s\subset a\}} + \mathbf{1}_{\{s\subset b\}} = 1 = \mathbf{1}_{\{s\subset a\cup b\}} + \mathbf{1}_{\{s\subset a\cap b\}}.
$$
Third, if $s\subset a$ and $s\subset b$, then 
$$
\mathbf{1}_{\{s\subset a\}} + \mathbf{1}_{\{s\subset b\}} = 2 = \mathbf{1}_{\{s\subset a\cup b\}} + \mathbf{1}_{\{s\subset a\cap b\}}.
$$

(b) It is clear that $f(a)$ is increasing. Also, to prove that 
$
f(a\cup b) + f(a\cap b) \geq f(a) + f(b)
$ 
for all $a,b \in\mathcal{L}$, we note that this inequality is equivalent to
$$
\sum_{s: \,s\subset (a\cup b)} n(s) + \sum_{s:\, s\subset (a\cap b)} n(s) \geq \sum_{s:\, s\subset a} n(s) + \sum_{s:\, s\subset b} n(s).
$$
Let $\sqcup$ denote the disjoint union. Note that
\begin{align*} 
\sum_{s:\, s\subset (a\cup b)} n(s) =& \sum_{s:\, s\subset (a\setminus b)} n(s) + \sum_{s:\, s\subset (a\cap b)} n(s) + \sum_{s:\, s\subset (b\setminus a)} n(s) + \sum_{\substack{s = s_1\sqcup s_2:\\s_1\subset (a\setminus b),\, s_2\subset (a\cap b)}} n(s)\\
&+\sum_{\substack{s = s_1\sqcup s_2:\\s_1\subset (b\setminus a),\, s_2\subset (a\cap b)}} \! n(s) + \sum_{\substack{s = s_1\sqcup s_2:\\s_1\subset (a\setminus b),\, s_2\subset (b\setminus a)}} \! n(s) +
\sum_{\substack{s = s_1\sqcup s_2\sqcup s_3:\\s_1\subset (a\setminus b),\, s_2\subset (a\cap b),\, s_3\subset (b\setminus a)}} \! n(s).
\end{align*}
Hence,
\begin{eqnarray*} 
\sum_{s:\, s\subset (a\cup b)} n(s) + \sum_{s:\, s\subset (a\cap b)} n(s)&=& \sum_{s:\, s\subset (a\setminus b)} n(s) + 2 \sum_{s:\, s\subset (a\cap b)} n(s) + \sum_{s:\, s\subset (b\setminus a)} n(s) \\
&&+ \sum_{\substack{s = s_1\sqcup s_2:\\s_1\subset (a\setminus b),\, s_2\subset (a\cap b)}} n(s) +\sum_{\substack{s = s_1\sqcup s_2:\\s_1\subset (b\setminus a),\, s_2\subset (a\cap b)}} n(s)\\ && + \sum_{\substack{s = s_1\sqcup s_2:\\s_1\subset (a\setminus b),\, s_2\subset (b\setminus a)}} n(s) +
\sum_{\substack{s = s_1\sqcup s_2\sqcup s_3:\\s_1\subset (a\setminus b),\, s_2\subset (a\cap b),\, s_3\subset (b\setminus a)}} n(s).
\end{eqnarray*}
Discarding the last two terms in the above sum, and rearranging the remaining terms, we obtain 
\begin{align*} 
\sum_{s:\, s\subset (a\cup b)} n(s) + \sum_{s:\, s\subset (a\cap b)} n(s) \ge & \sum_{s:\, s\subset (a\setminus b)} n(s) + \sum_{s:\, s\subset (a\cap b)} n(s) + \sum_{\substack{s = s_1\sqcup s_2:\\s_1\subset (a\setminus b),\, s_2\subset (a\cap b)}} n(s)\\
& +\sum_{s:\, s\subset (b\setminus a)} n(s) + \sum_{s:\, s\subset (a\cap b)} n(s) + \sum_{\substack{s = s_1\sqcup s_2:\\s_1\subset (b\setminus a),\, s_2\subset (a\cap b)}} n(s)\\
= & \sum_{s:\, s\subset a} n(s) + \sum_{s:\, s\subset b} n(s).
\end{align*}
This establishes the supermodularity of $f(a)$.  
\end{proof}

\begin{cor}\label{cor:incr_supermodular}
Let $g:\mathcal{L}\to \mathbb{R}$ be a non-negative function. Then the function 
$$
h(a) := \sum_{s\subset\mathcal{L}} \mathbf{1}_s(a) \,g(s) = \sum_{s: s\subset a} g(s),
$$
$a\in\mathcal{L}$, is increasing and supermodular.
\end{cor}

\begin{proof}
This is a consequence of Proposition \ref{prop:incr_supermodular}(b) because, in proving that result, we only used the property that
$$
\sum_{\substack{s = s_1\sqcup s_2:\\s_1\subset (a\setminus b),\, s_2\subset (b\setminus a)}} n(s) \;+
\sum_{\substack{s = s_1\sqcup s_2\sqcup s_3:\\s_1\subset (a\setminus b),\, s_2\subset (a\cap b),\, s_3\subset (b\setminus a)}} n(s) \;\geq\; 0,
$$
which completes the proof.
\end{proof}

\bigskip

\begin{remark}\label{correlationinequalities}
The supermodularity property can be applied to construct log-super\-modular probability distributions and to derive correlation inequalities for those distributions:  Let $\Theta \subset (\mathbb{R}_{\geq 0})^d$ be a parameter space, and define an exponential family probability distribution parametrized by $\theta\in\Theta$ on the lattice $\mathcal{L}$ with probability density function, 
\begin{equation}
\label{exp_measure}
\mu_{\theta}(a) = \exp\big(\theta^T n(a) - c(\theta)\big),
\end{equation}
$a \in \mathcal{L}$, where $\exp(-c(\theta))$ is the normalizing constant.  Then the probability distribution $\mu_{\theta}$ is \emph{log-supermodular}, i.e.,
$$
\mu_{\theta}(a \cup b) \mu_{\theta}(a \cap b) \ge \mu_{\theta}(a) \mu_{\theta}(b)
$$
for all $a, b \in \mathcal{L}$.  Log-supermodular distributions are tightly connected to distributions that are \emph{multivariate totally positive of order $2$} (MTP$_2$) \cite{karlin1980classes} (also known as \emph{FKG} \cite{fortuin1971correlation}); namely, a distribution $\mu_{\theta}$ on a lattice $\mathcal{L}$ is MTP$_2$ if 
$$\mu_{\theta}(a \cup b) + \mu_{\theta}(a \cap b) \ge \mu_{\theta}(a) + \mu_{\theta}(b)$$
for all $a, b \in \mathcal{L}$. Hence for strictly positive distributions, log-supermodularity and MTP$_2$ are equivalent. 
Note that the MTP$_2$ property depends on the labeling of the points in the lattice. For example, to check if the sample distribution in Table~\ref{leadstudy} is MTP$_2$, we need to check
$$n_{i,j} + n_{k,l} \leq n_{\min(i,k), \min(j,l)} + n_{\max(i,k), \max(j,l)}$$
for all $i,j,k,l\in\{1,2,3\}$. This leads to eight non-trivial inequalities. When encoding the father's hygiene (poor, medium, low) by (1,2,3) and the father's exposure (low, medium, high) by (1,2,3), then one can check that the sample distribution is not MTP$_2$, since for example
$$n_{2,1} + n_{1,3} \nleq n_{1,1} + n_{2,3}.$$
However, if we encode the father's hygiene (poor, medium, low) by (3,2,1) and the father's exposure (low, medium, high) by (2,1,3), then one can easily check that all eight non-trivial inequalities are satisfied and that the distribution is MTP$_2$. Distributions that are MTP$_2$ up to a relabeling of the states were studied in~\cite{non_negative_tensors}.

Note that the FKG inequality  \cite{fortuin1971correlation} can be used to obtain interesting correlation inequalities on contingency tables: Let $h_1$ and $h_2$ be decreasing functions on the lattice $\mathcal{L}$; then, by the FKG inequality, the covariance, 
$
\Cov(h_1,h_2) := \E(h_1 h_2) - \E(h_1) E(h_2)
$ 
is nonnegative. For example, with $h_1(a) = n(a \cap \alpha)$ and $h_2(a) = n(a \cap \beta)$, $a \in \mathcal{L}$, it follows from the FKG inequality that the cell entries in the marginal table of $a \cap \alpha$ are positively correlated with the cell entries in the marginal table of $a \cap \beta$.  

One can also construct more general log-supermodular probability distributions and derive correlation inequalities for those models, as was done in \cite{rosenbaum2002observational}. For example, let
\begin{equation}
\label{exp_measure-interaction}
\mu_{\theta}(a) = \exp\big(\theta_1^T n(a) + \theta_2^T n(a \cap \alpha) - c(\theta)\big),
\end{equation}
with parameter $\theta:=(\theta_1, \theta_2)\in(\mathbb{R}_{\geq 0})^d$, where $a,\alpha\in \mathcal{L}$. This log-supermodular density function is related to exponential family models arising in observational studies \cite[Chapter 4]{rosenbaum2002observational} and to Ising and Potts models arising in graphical models \cite[Subsection 3.3]{wainwright2008graphical}. 
\end{remark}

\section{Applications of an inequality by Ky Fan}\label{sec:fan}

Fan~\cite{fan1968inequality} derived a remarkable inequality for supermodular functions. We will show that many known bounds on the cell entries of a multidimensional contingency table follow from Fan's inequality.  We will also derive new inequalities from Fan's inequality and, further, we will discuss an example of bounds on the cell entries that do not follow from Fan's inequality.

\begin{thm}[Fan~\cite{fan1968inequality}]
\label{thm_Ky_Fan}
Suppose that $f$ is a supermodular function defined on a distributive lattice $\mathcal{L}$.  Then for any finite sequence $x_1,\dots,x_q$ of elements in $\mathcal{L}$, we have
\begin{equation}
\label{eq:fan1}
\sum_{1\leq i_1<\cdots <i_p \leq q} f(x_{i_1}\wedge \cdots \wedge x_{i_p}) \quad\leq\quad \sum_{k=p}^q \binom{k-1}{p-1} \; f\Big(\bigvee_{1\leq i_1<\cdots <i_k \leq q} (x_{i_1}\wedge \cdots \wedge x_{i_k})\Big),
\end{equation}
$1 \le p \le q$, and dually,
\begin{equation}
\label{eq:fan2}
\sum_{1\leq i_1<\cdots <i_p \leq q} f(x_{i_1}\vee \cdots \vee x_{i_p}) \quad\leq\quad \sum_{k=p}^q \binom{k-1}{p-1} \; f\Big(\bigwedge_{1\leq i_1<\cdots <i_k \leq q} (x_{i_1}\vee \cdots \vee x_{i_k})\Big).
\end{equation}
\end{thm}

Fan's proof of the inequality (\ref{eq:fan1}) is by induction, as follows: First, the case in which $p=1$ and $q \ge 1$ is established by induction on $q$.  Next, it is noted that if $q=p$ then both sides of the inequality (\ref{eq:fan1}) are identically equal to $f(x_1\wedge \cdots \wedge x_q)$.  Finally, for $2 \le p < q$, it is shown by induction that the case $(p,q)$ follows from the cases $(p,q-1)$ and $(p-1,q-1)$.

\begin{example}
\label{example:3waytables}
This example demonstrates the use of Fan's inequality (\ref{eq:fan1}) to derive Fr\'echet bounds arising in the analysis of 3-way contingency tables.  Fienberg~\cite[Section 6]{fienberg1999frechet} provided Fr\'echet bounds based on the 1-dimensional marginals, namely, 
\begin{eqnarray*}
\min\big(n(\{1\}), n(\{2\}),n(\{3\})\big) &\geq& n(\{1,2,3\})\\
&\geq& \max\big(n(\{1\}) +  n(\{2\}) + n(\{3\})  -2 n(\emptyset), 0\big),
\end{eqnarray*}
and the bounds based on the 2-dimensional marginals, namely
\begin{align*}
\min\big(n(\{1,2\}),&n(\{1,3\}),n(\{2,3\})\big) \\
&\;\geq\; n(\{1,2,3\})\\
&\geq\; \max\big(n(\{1,2\}) + n(\{1,3\}) - n(\{1\}), n(\{1,2\}) + n(\{2,3\}) - n(\{2\}),\\
&\qquad\qquad n(\{1,3\}) + n(\{2,3\})  - n(\{3\}), 0\big).
\end{align*}

Note that the upper bounds are a consequence of the fact that $n(\cdot)$ is decreasing. The lower Fr\'echet bound based on the 1-dimensional marginals follows from Fan's inequality (\ref{eq:fan1}) by taking $p=1$, $q=3$ and $x_i=\{i\}$ for $i=1,2,3$. The lower Fr\'echet bound based on the 2-way marginals follows from Fan's inequality by taking $p=1$, $q=2$ and taking for the $x_i$'s two sets of two elements such as $x_1=\{1,2\}$ and $x_2=\{1,3\}$.
\end{example}

We now discuss certain generalized Fr\'echet inequalities described by Fienberg~\cite{fienberg1999frechet}. Fr\'echet bounds based on the 1-dimensional marginals can be found in a variety of sources (see, e.g., \cite[Equation (6)]{fienberg1999frechet} and \cite{kwerel1988frechet,walter1988marginal,ruschendorf1991bounds}) and are as follows:
$$
\min\big(n(\{1\}),\dots,n(\{\ell\})\big) \;\geq\; n(\{1,\dots,\ell\})\;\geq\; \max\left(\sum_{j=1}^\ell n(\{j\}) -(\ell-1) n(\emptyset), 0\right).
$$ 
Note that the first inequality is a consequence of the property that $n(\cdot)$ is decreasing, and the second inequality is a corollary of Fan's inequality (\ref{eq:fan1}) with $p=1$, $q=\ell$, and $x_{i} = \{i\}$, $i=1,\ldots,\ell$.  Using Fan's inequality, we now generalize these $1$-dimensional Fr\'echet bounds to any dimension $d$, where $1 \leq d \leq \ell$.

\begin{cor}\label{cor:d_frechet_bounds}
Let $n$ be an $\ell$-way contingency table and let $1\leq d\leq \ell$. Then
\begin{equation}
\begin{aligned}
\label{one_dim_Frechet}
\min\big(n(\{&j_1,\dots,j_d\}) \, :\, 1 \leq j_1 < \cdots <j_d\leq \ell\big) \\
& \;\geq\; n(\{1,\dots,\ell\}) \\
& \;\geq\; \max\left(\frac{1}{\binom{\ell-1}{d-1}}\,\sum_{1\leq j_1 <\cdots <j_d\leq \ell} n(\{j_1,\dots,j_d\}) - \Big(\frac{\binom{\ell}{d}}{\binom{\ell-1}{d-1}} - 1\Big)n(\emptyset), 0\right).
\end{aligned}
\end{equation}
\end{cor}

\begin{proof}
The first inequality follows from the fact that $n(\cdot)$ is decreasing. 

The second inequality follows from Fan's inequality (\ref{eq:fan1}) with $p=1$ and $q = \binom{\ell}{d}$, as follows: Let $J = \{(j_1,\ldots,j_d): 1 \le j_1 < \cdots < j_d \le \ell\}$ be the set of all subsets of size $d$ chosen from $\{1,\ldots,\ell\}$.  Then, 
$$
n\left(\bigvee_{\{i_1,\ldots,i_k\} \subset J} (x_{i_1}\wedge \cdots \wedge x_{i_k})\right) \, = \, 
\begin{cases}
n(\{1,\dots,\ell\}), & \textrm{if } k \leq \binom{\ell-1}{d-1} \\
n(\emptyset), & \textrm{otherwise}.
\end{cases}
$$
Applying Fan's inequality completes the proof.
\end{proof}

We remark that the lower bound in (\ref{one_dim_Frechet}) is a generalized Fr\'echet or generalized Boole inequality; see Kwerel \cite[ Eqs.~(13) and (14)]{kwerel1988frechet}.  In Kwerel's notation, 
\begin{equation}
\label{kwerel1}
\frac{n(\{1,\dots,\ell\})}{n(\emptyset)} \equiv p_{1,\ldots,\ell},
\end{equation}
and 
\begin{equation}
\label{kwerel2}
\frac{1}{n(\emptyset)} \sum_{1 \leq j_1 < \cdots < j_d \leq \ell} n(\{j_1,\dots,j_d\}) \equiv S_d.
\end{equation}
Dividing the lower bound in (\ref{one_dim_Frechet}) by $n(\emptyset)$, we obtain 
\begin{equation}
\label{kwerel3}
\frac{n(\{1,\dots,\ell\})}{n(\emptyset)} \ge \frac{1}{\binom{\ell-1}{d-1}}\,\frac{1}{n(\emptyset)} \sum_{1\leq j_1 <\cdots <j_d\leq \ell} n(\{j_1,\dots,j_d\}) - \frac{\binom{\ell}{d}}{\binom{\ell-1}{d-1}} + 1.
\end{equation}
Noting that 
$$
\frac{\binom{\ell}{d}}{\binom{l-1}{d-1}} = \frac{l}{d},
$$
it follows from (\ref{kwerel1}) and (\ref{kwerel2}) that (\ref{kwerel3}) is equivalent to 
$$
p_{1,\ldots,\ell} \ge \frac{S_d}{\binom{\ell-1}{d-1}} - \frac{l}{d} + 1,
$$
which is an inequality stated by Kwerel.  

\medskip

We now generalize a Fr\'echet-type inequality given by Dobra and Fienberg \cite[Theorem 6]{dobra2000bounds}.  Our proof also reveals that the inequality requires no graph-theoretic hypotheses, so that it holds in general.  

\begin{thm}\label{thm:nographs}
Let $C_1,\dots,C_d \in \mathcal{L}$ with $C_1\cup\cdots\cup C_d=\{1,\dots,\ell\}$. Define $S_j = (C_1\cup\cdots\cup C_{j-1})\cap C_j$, $j=2,\ldots,d$. Then
$$
\min\big(n(C_1),\dots,n(C_d)\big) \;\geq\; n(\{1,\dots,\ell\})\;\geq\; \max\left(\sum_{i=1}^d n(C_i) -\sum_{j=2}^d n(S_j), 0\right).
$$
\end{thm}

\begin{proof}
The first inequality follows from the fact that $n(\cdot)$ is decreasing. 

The proof of the second inequality is by induction on $d$. For $d=2$, the claim follows from the supermodularity property:
\begin{align*}
n(\{1,\dots,\ell\}) &\;\equiv\; n(C_1\cup C_2) \;\geq\; n(C_1)+n(C_2) - n(C_1\cap C_2) \;\equiv\; \sum_{i=1}^2 n(C_i) - n(S_2).
\end{align*}
Now suppose that the claim holds for the sets $C_1,\ldots,C_{d-1}$. Then by supermodularity,
\begin{eqnarray*}
n(\{1,\dots,\ell\}) &=& n\big((C_1\cup \cdots \cup C_{d-1}) \cup C_d\big) \\
&\geq& n(C_1\cup \cdots \cup C_{d-1}) + n(C_d) - n\big((C_1\cup \cdots \cup C_{d-1}) \cap C_d\big) \\
&=& n(C_1\cup \cdots \cup C_{d-1}) + n(C_d) - n(S_d).
\end{eqnarray*}
By the inductive hypothesis,
$$
n(C_1\cup \cdots \cup C_{d-1})\quad\geq\quad \sum_{i=1}^{d-1} n(C_i) -\sum_{j=2}^{d-1} n(S_j)
$$
and hence
$$
n(\{1,\dots,\ell\})\quad\geq\quad \sum_{i=1}^d n(C_i) -\sum_{j=2}^d n(S_j),
$$
which establishes the claim.
\end{proof}

It is interesting that although the proof is by induction, the result does not appear to follow from Fan's inequality which, as we observed before, is also derived by induction.  For example, for $d=3$ Fan's inequality with $p=1$ provides 
\begin{eqnarray*}
n(C_1)+n(C_2)+n(C_3)&\leq& n(C_1\cup C_2\cup C_3) + n((C_1\cap C_2)\cup (C_1\cap C_3)\cup (C_2\cap C_3))\\
&& +\, n(C_1\cap C_2\cap C_3)\\
&=& n(C_1\cup C_2\cup C_3) + n(S_2\cup S_3) + n(S_2\cap S_3).
\end{eqnarray*}
However, by supermodularity, 
$$
n(C_1\cup C_2\cup C_3) + n(S_2\cup S_3) + n(S_2\cap S_3) \;\geq\; n(C_1\cup C_2\cup C_3) + n(S_2)+n(S_3),
$$
and hence Fan's inequality results in a bound which is weaker than the inequality derived in Theorem~\ref{thm:nographs}.

\section{Discussion}\label{sec:discussion}

These considerations lead to a general approach to constructing families of Fr\'echet-type inequalities.  Starting with $f$, a log-supermodular strictly positive density function on $\mathcal{L}$, we construct $g = \log f$, a supermodular nonnegative function and then apply Fan's inequalities to $g$ and interpret those inequalities as Fr\'echet-type inequalities.  For example, we obtain the original Fr\'echet inequalities by choosing the log-supermodular density function given in (\ref{exp_measure}). Bearing in mind the many available examples of log-supermodular density functions \cite{karlin1980classes}, this procedure leads to a variety of inequalities.  


\section*{Acknowledgments}

We thank Milan Studen\'y for his invitation to D.R. to speak at the Workshop on ``Limit Theorems and Algebraic Statistics,'' Prague Stochastics 2014, August 25--29, 2014, held at the Institute of Information Theory and Automation, Academy of Sciences of the Czech Republic, Prague, where parts of this work were first presented. We also thank the two anonymous referees for their helpful feedback.

C.U.'s research was partially supported by DARPA (W911NF-16-1-0551), NSF (DMS-1651995), ONR (N00014-17-1-2147), and a Sloan Fellowship. D.R.'s research was partially supported by the U.S.~National Science Foundation grant DMS-1309808; by a 2013--2014 sabbatical leave-of-absence at Heidelberg University; and by a Romberg Guest Professorship at the Heidelberg University Graduate School for Mathematical and Computational Methods in the Sciences, funded by German Universities Excellence Initiative grant GSC 220/2.

\bibliographystyle{abbrv}
\bibliography{generalized_Frechet}

\end{document}